\documentclass[a4paper,11pt]{amsart}

\usepackage{epsfig}
\usepackage{amsmath,amsthm,amsfonts,amssymb}
\usepackage{graphicx,color}
\usepackage[all]{xy}
\usepackage{enumerate}
\usepackage{verbatim}
\usepackage{hyperref}

\newtheorem{theorem}{Theorem}[section]
\newtheorem*{theorem*}{Theorem}
\newtheorem*{theorem_A}{Theorem A}
\newtheorem*{theorem_B}{Theorem B}
\newtheorem*{theorem_C}{Theorem C}
\newtheorem{lemma}[theorem]{Lemma}
\newtheorem{corollary}[theorem]{Corollary}
\newtheorem{proposition}[theorem]{Proposition}
\newtheorem{definition}[theorem]{Definition}

\newtheorem{conjecture}[theorem]{Conjecture}

\DeclareMathOperator{\Div}{div}

\DeclareMathOperator{\ric}{Ric}

\begin{document}

\title[Bounds to the mean curvature of leaves of CMC foliations]{Bounds to the mean curvature of leaves of CMC foliations}

\author[J. E. Sampaio]{Jos\'e Edson Sampaio}
\author[E. C. da Silva]{Euripedes Carvalho da Silva}

\address{Jos\'e Edson Sampaio:  
	Departamento de Matem\'atica, Universidade Federal do Cear\'a,
	Rua Campus do Pici, s/n, Bloco 914, Pici, 60440-900, 
	Fortaleza-CE, Brazil. \newline  
	E-mail: {\tt edsonsampaio@mat.ufc.br}                    
}
\address{Euripedes Carvalho da Silva: Departamento de Matem\'atica, Instituto Federal de Educa\c{c}\~ao, Ci\^encia e Tecnologia do Cear\'a,
 	      Av. Parque Central, 1315, Distrito Industrial I, 61939-140, 
 	      Maracana\'u-CE, Brazil. \newline 
 	      and 	      Departamento de Matem\'atica, Universidade Federal do Cear\'a,
	      Rua Campus do Pici, s/n, Bloco 914, Pici, 60440-900, 
	      Fortaleza-CE, Brazil.
               E-mail: {\tt euripedes.carvalho@ifce.edu.br}
 } 

\thanks{The second-named author was partially supported by CNPq-Brazil grant 310438/2021-7. 
This work was supported by the Serrapilheira Institute (grant number Serra -- R-2110-39576).}


\begin{abstract}
The main goal of this present paper is to bring the results  proved by Barbosa, Kenmotsu and Oshikiri (1991) and its ideas to a perspective where the Ricci curvature is bounded from below. For instance, for a foliation by CMC hypersurfaces on a compact (without boundary) Riemannian manifold $M^{n+1}$ with Ricci curvature bounded from below by $-nK_0\leq 0$ and such that the mean curvature $H$ of the leaves of the foliation satisfies $|H|\geq \sqrt{K_0}$, we prove that $|H|\equiv \sqrt{K_0}$ and all the leaves are totally umbilical. This gives, in particular, a generalization for the result proved by Barbosa, Kenmotsu and Oshikiri (1991), where the above result was proved in the case $K_0=0$. We also obtain a proof of the following: for a foliation by CMC hypersurfaces on a compact (without boundary) Riemannian manifold $M$ with Ricci curvature bounded from below by $-nK_0\leq 0$, the mean curvature $H$ of the leaves of the foliation satisfies $|H|\leq \sqrt{K_0}$. Furthermore, if the foliation contains a leaf $L$ whose absolute mean curvature is $|H_L|=\sqrt{K_0}$, then either $K_0=0$ and all the leaves of $\mathfrak{F}$ are totally geodesic, or $K_0>0$ and there is a totally umbilical leaf. 
\end{abstract}
\maketitle
\vspace{-0.5cm}
\tableofcontents
\section{Introduction}

In 1991, in the very interesting article \cite{BarbosaKO-1991}, Barbosa, Kenmotsu and Oshikiri proved the following results:

\begin{theorem}[Theorem 3.1 in \cite{BarbosaKO-1991}]
Let $\mathfrak{F}$ be a codimension one CMC foliation of a compact (without boundary) Riemannian manifold $M^{n+1}$ with non-negative Ricci curvature. Then any leaf of $\mathfrak{F}$ is a totally geodesic submanifold of $M$. Furthermore, $M$ is locally a Riemannian product of a leaf of $\mathfrak{F}$ and a normal curve, and the Ricci curvature in the direction normal to the leaves is zero.
\end{theorem}

\begin{theorem}[Theorem 3.8 in \cite{BarbosaKO-1991}]
Let $\mathfrak{F}$ be a transversely oriented foliation of $Q^{n+1}(c)$ such that
each leaf $L$ has constant mean curvature $H_L$, $Q^{n+1}(c)$ represents a $(n+1)$-dimensional complete Riemannian manifold with constant sectional curvature $c$. Assume $c\leq 0$ and $|H_L|\geq \sqrt{-c}$. Then, $\inf\limits_{p\in M}|H(p)|= \sqrt{-c}$.
\end{theorem}

The main goal of this present paper is to bring the results in \cite{BarbosaKO-1991} and its ideas
to a perspective where the Ricci curvature is bounded from below, besides those already presented in \cite{BarbosaBM-2008}. In particular, we obtain the following results.

\begin{theorem_A}
Let $\mathfrak{F}$ be a codimension one CMC foliation of a complete oriented Riemannian manifold $M^{n+1}$ with Ricci curvature bounded from below by $-nK_0\leq 0$. Suppose that the mean curvature function $H\colon M\rightarrow \mathbb{R}$, which associates to each point the value of the mean curvature of the leaf of $\mathfrak{F}$ that contains that point, does not change sign on $M$. Then, $\sqrt{K_0}\geq \inf\limits_{p\in M}|H(p)|$.
\end{theorem_A}

\begin{theorem_B}
Let $\mathfrak{F}$ be a codimension one foliation of a compact (without boundary) Riemannian manifold $M^{n+1}$. Assume that there is $K_0\geq 0$ such that ${\rm Ric}\geq -nK_0$. Suppose that each leaf $L$ of $\mathfrak{F}$ has constant mean curvature $H_L$ such that $|H_L|\geq \sqrt{K_0}$. Then $|H_L|\equiv \sqrt{K_0}$ for any leaf $L$ of $\mathfrak{F}$ and each leaf of $\mathfrak{F}$ is totally umbilical. 
\end{theorem_B}

By using the results in \cite{Oshikiri-1981}, we recover the full statement of Theorem 3.1 in \cite{BarbosaKO-1991}.

In 2008, Meeks III, P\'erez and Ros in \cite[Conjecture 5.1.2]{MeeksPR:2008} proposed the following conjecture:
\begin{conjecture}\label{conjecture:control_mean_curv}
Let $\mathfrak{F}$ be a codimension one foliation of a complete Riemannian manifold $M^{n+1}$. Assume that $M$ has absolute sectional curvature bounded from above by $1$. Suppose that each leaf $L$ of $\mathfrak{F}$ has constant mean curvature $H_L$.
Then $|H_L|\leq 1$.
\end{conjecture}
Meeks III, P\'erez and Ros in \cite[Corollary 5.10]{MeeksPR:2008} proved Conjecture \ref{conjecture:control_mean_curv} has a positive answer in the case that $M=\tilde M^3(-1)$. They also proved in \cite[Theorem 5.23]{MeeksPR:2008} that when $M$ is a homogeneously regular manifold with absolute sectional bounded from above by 1 and $n=3$ or $4$, the absolute mean curvature of any leaf of a codimension one CMC foliation of $M$ is bounded by some constant $H_n$ that only depends on $n$.

Here, we consider the following more general conjecture:
\begin{conjecture}\label{conjecture:control_mean_curv_general}
Let $\mathfrak{F}$ be a codimension one foliation of a complete Riemannian manifold $M^{n+1}$. Assume that there is $K_0\geq 0$ such that $M$ has Ricci curvature bounded from below by $-nK_0$. Suppose that each leaf $L$ of $\mathfrak{F}$ has constant mean curvature $H_L$.
Then $|H_L|\leq \sqrt{K_0}$.
\end{conjecture}
Barbosa, Kenmotsu and Oshikiri in \cite{BarbosaKO-1991} gave a positive answer to this conjecture when $M$ is a compact (without boundary) manifold and $K_0=0$.
This conjecture was positively answered by Meeks III, P\'erez and Ros in \cite[Theorem 5.8.C]{MeeksPR:2008} when $M$ is a compact (without boundary) orientable $3$-manifold, which is not topologically covered by $\mathbb{S}^2\times \mathbb{S}^1$. Indeed, in this case they prove even more, instead of the assumption that the Ricci curvature is bounded from below by $-2K_0$, they only need that the scalar curvature of $M$ is bounded from below by $-6 K_0$.
For any dimension, but still when $M$ is a compact (without boundary) manifold, a positive answer to Conjecture \ref{conjecture:control_mean_curv_general} follows from the Structure Theorem for CMC Foliations proved by Meeks III and P\'erez in \cite{MeeksPR:2016b} (see \cite[Remark 1.4.i]{MeeksPR:2016b}). Here, we present another proof of this fact with the following slightly more general statement:
\begin{theorem_C}
Let $\mathfrak{F}$ be a codimension one CMC foliation of a compact (without boundary) Riemannian manifold $M^{n+1}$. Assume that there is $K_0\geq 0$ such that the Ricci curvature of $M$ is bounded from below by $-nK_0$. Then, for each leaf $L$ of $\mathfrak{F}$, the mean curvature $H_L$ of $L$ satisfies $|H_L|\leq \sqrt{K_0}$. Furthermore, if $\mathfrak{F}$  contains a leaf $L$ whose absolute mean curvature is $|H_L|=\sqrt{K_0}$, then either $K_0=0$ and all the leaves of $\mathfrak{F}$ are totally geodesic, or $K_0>0$ and there is a totally umbilical leaf in $\mathfrak{F}$.
\end{theorem_C}

\section{Preliminaries}\label{sec:preliminaries}

In this Section, we introduce some basic facts and notations that will appear in the paper.

Here, the Riemannian manifolds are assumed to be connected and without boundary and the foliations are assumed to be $C^2$ smooth.

Let $M^{n+1}$ be an $(n+1)$-dimensional Riemannian manifold endowed with a {\bf Riemannian metric $g_M=\sum{\omega_A^2}$} and  $\mathfrak{F}$ is a foliation of codimension one on $M$.

For a given point $p\in M$ we can choose an orthonormal frame $\{e_1,\cdots,e_n,e_{n+1}\}$ defined around $p$ such that the vectors $e_1,\cdots,e_n$ are tangent to the leaves of $\mathcal{F}$ and $e_{n+1}$ is normal to them. Taking the correspondent dual coframe
$$\{\omega_1,\cdots,\omega_n,\omega_{n+1}\},$$ 
the  {\bf structure equations} on  $M$ are given by 

\begin{equation}
	d\omega_A=\sum_{B=1}^{n+1}{\omega_B\wedge \omega_{BA}}, \ \ \omega_{AB}+\omega_{BA}=0
\end{equation}

\begin{equation}
	d\omega_{AB}=\sum_{C=1}^{n+1}{\omega_{AC}\wedge \omega_{CB}}+\Omega_{AB},
\end{equation}
where
\begin{equation}
	\Omega_{AB}=-\frac{1}{2}\sum_{C,D=1}^{n+1}{ R_{ABCD}\omega_{C}\wedge \omega_{D}}, \ \ R_{ABCD}+R_{ABDC}=0.
\end{equation}
The {\bf Ricci curvature} in the direction  $e_{n+1}$  is
\begin{equation}
	\ric(e_{n+1})=\sum_{i=1}^{n}{g_M(R(e_{n+1}, e_i) e_{n+1},e_i)}.
\end{equation}

Let  $\nabla$ be the {\bf Levi-Civita connection} on $M$.
Then for any tangent field $X$, we get
\begin{equation}
	\nabla_X{e_A}=\sum_{B=1}^{n+1}{\omega_{AB}(X) e_B}.
\end{equation}

Now let $\theta_A$ and $\theta_{AB}$ denote the restrictions of forms $\omega_A$ and $\omega_{AB}$ to the tangent vectors of the leaves of $\mathcal{F}$. Then it is obvious that
\begin{equation}
	\theta_{n+1}=0 \ \ \mbox{e} \ \ \theta_{i}=\omega_{i}.
\end{equation}

Since $\theta_{n+1}=0$, we obtain from the structure equations  

$$0=d\theta_{n+1}=\sum_{B=1}^{n+1}{\theta_B\wedge \theta_{B n+1}}=\sum_{i=1}^n{\theta_{i}\wedge \theta_{i n+1}}.$$

By {\bf Cartan's equation}, we have
\begin{equation}
	\theta_{n+1i}=-\sum_{j=1}^{n}{h_{ij}\theta_j}, \ \ h_{ij}=h_{ji}.
\end{equation}

The {\bf second fundamental form  $\mathrm{I\!I}$} of the leaves is then given by
\begin{equation}
	\mathrm{I\!I}=\sum_{i=1}^n{\theta_i\otimes\theta_{in+1}}=\sum_{i,j=1}^{n}{h_{ij}\theta_i\otimes\theta_{j}}.
\end{equation}

and its norm is 
$$\left\|\mathrm{I\!I}\right\|^2=\sum_{i,j=1}^{n} h_{ij}^{2}.$$

The {\bf mean curvature vector} is  

\begin{equation}
	\vec{H}=\frac{1}{n}tr(A)e_{n+1},
\end{equation}
where  $A$ is the {\bf Weingarten operator} and the {\bf mean curvature  function} is  $H=\frac{1}{n}tr(A)$.

Observe that the sign of $H$ depends on the choice of $e_{n+1}$. If $N$ is a unitary vector field normal to the leaves of $\mathcal{F}$, we can choose an adapted frame on an open set in such a way that $N=e_{n+1}$. The mean curvature of the leaf is exactly the mean curvature in the direction of $N$.

The  {\bf divergent of a vector field}  $V$ is defined  locally over  $M$ by 
\begin{equation}
	\Div(V)=\sum_{A=1}^{n+1}{g_M(\nabla_{e_A}V,e_A)}.
\end{equation}
For a vector field tangent to the leaves of  $\mathcal{F}$ the divergent along the leaves can be computed by
\begin{equation}
	\Div_L(V)=\sum_{i=1}^{n}{g_M(\nabla_{e_i}V,e_i)}.
\end{equation}
Barbosa et al. in \cite{BarbosaKO-1991} found an equation that relates the foliation  with the ambient, more precisely, they obtained the following.
\begin{proposition} \label{propequation}
	Let $\mathcal{F}$ be a codimension one foliation by hypersurfaces on a Riemannian manifold $M$ and let $N$ be a unit field normal to the leaves of $\mathcal{F}$ on some open set $U$ of $M$. Then on $U$, we have
	\begin{subequations}
		\begin{align}
			\label{divN} \Div N & = -nH;\\
			\label{divL}  \Div_L(X) & =-nN(H)+\left\|\mathrm{I\!I}\right\|^2+\ric(N)+\|X\|^2;\\
			\label{divX} \Div X & = \Div_L X -\|X\|^2,
		\end{align}
	\end{subequations}
	where $H$ is the mean curvature in the direction $N$ and $X=\nabla_N N$.
\end{proposition}
%


\begin{definition}
Let $M$ be a Riemannian manifold. We say that $\mathfrak{F}$ is a {\bf foliation of constant mean curvature} (or {\bf CMC foliation}) on $M$ if each leaf $L\in \mathfrak{F}$ is a hypersurface of constant mean curvature (note that the mean curvature possibly varies from leaf to leaf). We say that the foliation $\mathfrak{F}$ is a
{\bf minimal foliation} if each leaf $L\in \mathfrak{F}$ is a minimal hypersurface. 

\end{definition} 

The {\bf Riemannian volume} in an $n$-dimensional Riemannian manifold is the $n$-dimensional Hausdorff measure determined by its Riemannian metric.

\begin{definition}
Let $M$ be a connected complete Riemannian manifold and $B(p,r)$ be the geodesic ball centred in $p\in M$ with radius $r$. 
The {\bf volume entropy of $M$} is 
$$
\mu_M=\limsup\limits_{r\to +\infty}\frac{\ln\mbox{vol}_M\left( B(p,r) \right)}{r},$$
 where $\mbox{vol}_M(B)$ denotes the Riemannian volume in $M$ of $B$. The {\bf lower volume entropy of $M$} is 
$$
\underline{\mu_M}=\liminf\limits_{r\to +\infty}\frac{\ln\mbox{vol}_M\left( B(p,r) \right)}{r}.$$
\begin{itemize}
 \item We say that $M$ has {\bf polynomial volume growth} if there are a point $p\in M$, a non-negative integer $d$ and $a,b>0$ such that 
$$\mbox{vol}_M\left( B(p,r) \right)\leq ar^d+b,$$
for all $r>0$.
\item We say that $M$ has {\bf zero volume entropy} or {\bf subexponential volume growth} if $\mu_M=0$.
\item We say that $M$ has {\bf zero lower volume entropy} if $\underline{\mu_M}=0$.
\end{itemize}

\end{definition}
Note that the choice of $p$ in the above concepts is irrelevant. Note also that ``polynomial volume growth'' $\Rightarrow$ ``zero volume entropy'' $\Rightarrow$ ``zero lower volume entropy''.

\section{The main results}

\subsection{Proof of Theorem A}
\begin{theorem_A}\label{thm:inf_mean_curv_bounded_ricci}
Let $\mathfrak{F}$ be a CMC foliation of a complete Riemannian manifold $M^{n+1}$ with Ricci curvature bounded from below by $-nK_0\leq 0$. Suppose that the mean curvature function $H\colon M\rightarrow \mathbb{R}$, which associates to each point the value of the mean curvature of the leaf of $\mathfrak{F}$ that contains that point, does not change sign on $M$. Then, $\sqrt{K_0}\geq \inf\limits_{p\in M}|H(p)|$.
\end{theorem_A}
\begin{proof}
We may assume that $M$ is a complete orientable Riemannian manifold and that $\mathfrak{F}$ is transversely orientable, otherwise we will work on the orientable double covering of $M$. We then have on $M$ a globally defined unit vector field $N$ normal to the leaves of $\mathfrak{F}$. 

By changing $N$ by $-N$ and $H$ by $-H$, if necessary, we may assume that $H\leq 0$. Let $c\in [0,+\infty)$ such that $\sup_{p\in M} H(p) = -c\leq 0$. Note that $c=\inf_{x\in M}{|H(x)|}$. 

Since $M$ is a complete Riemannian manifold the flow $\theta_t\colon M\rightarrow M$ of the normal vector field $N$ of the foliation is globally defined.
	
	Thus, we define the smooth function $\varphi\colon [0,+\infty)\rightarrow (0,+\infty)$ given by  $$\varphi(t)=\mbox{vol}_M\left(\theta_{t}(B)\right)=\int_{\theta_{t}(B)}dM=\int_{B}\theta^{*}_{t}dM,$$
	where $B:=B(p,r)$ is the geodesic ball centred at $p$ and radius $r$.  
By proceeding like in the proof of Lemma \ref{lemma:volpol}, we obtain that 
	$$\mbox{vol}_M(B(p,t+r))\geq \mbox{vol}_M(\theta_{t}(B))=\varphi(t)=\mbox{vol}_M(B)e^{nc t},\,\,\,\forall \,\,t\geq0.$$
By Bishop–Gromov inequality, we obtain
$$
\mbox{vol}_M(B(p,s))\leq \mbox{vol}_{\tilde M^{n+1}(-K_0)}(B_{\tilde M^{n+1}(-K_0)}(\tilde p,s)),
$$
for all $s>0$, where ${\tilde M^{n+1}(-K_0)}$ is the space form of constant sectional curvature $-K_0$ and $B_{\tilde M^{n+1}(-K_0)}(\tilde p,s)$ is the geodesic ball of ${\tilde M^{n+1}(-K_0)}$ centred at $\tilde p$ and of radius $s$. However, we have that
$$
\mbox{vol}_{\tilde M^{n+1}(-K_0)}\left(B_{\tilde M^{n+1}(-K_0)}(\tilde p,s)\right)=c_n\int_0^s \left(\frac{\sinh(\sqrt{K_0}t)}{\sqrt{K_0}}\right)^ndt,
$$
where $c_n$ is the $n$-dimensional volume of the unit sphere in $\mathbb{R}^{n+1}$ (see \cite[p. 105]{BarbosaKO-1991} or \cite[\S III.4.1]{Chavel:2006}). Thus, by using the L'Hospital rule, we obtain
$$
n\sqrt{K_0}=\lim\limits_{s\to +\infty}\frac{\ln\mbox{vol}_{\tilde M^{n+1}(-K_0)}\left(B_{\tilde M^{n+1}(-K_0)}(\tilde p,s)\right)}{s}.
$$
Now, we need the following result:
\begin{lemma}\label{lemma:volpol}
Let $\mathfrak{F}$ be a transversely oriented codimension one foliation on a complete oriented Riemannian manifold $M^{n+1}$. Let $H\colon M\rightarrow \mathbb{R}$ be the mean curvature function, which associates to each point the value of the mean curvature of the leaf of $\mathfrak{F}$ through that point. Then, $\underline{\mu_M}\geq n \inf\limits_{p\in M}|H(p)|$. Furthermore, if $M^{n+1}$ has zero lower volume entropy, then $\inf\limits_{p\in M}|H(p)|=0$.
\end{lemma}	
\begin{proof}[Proof of Lemma \ref{lemma:volpol}]
	
	Since $\mathfrak{F}$ is a transversely oriented foliation on a complete oriented Riemannian manifold $M^{n+1}$, and the mean curvature function $H\colon M\rightarrow \mathbb{R}$, that associates to each point the value of the mean curvature of the leaf of $\mathfrak{F}$ that contains that point, does not change sign on $M$, we can choose the normal vector field $N$ or $-N$ in such a way that $H\leq 0$. Since $M$ is a complete Riemannian manifold the flow $\theta_t\colon M\rightarrow M$ of the normal vector field $N$ of the foliation is globally defined. 
		
	Now, following classical arguments and standard results on geometric flows (see Theorem 5.20 in \cite{GallotHL:2004}, Proposition 16.33 in \cite{Lee:2003}) and following closely the ideas of the proof of Theorem 2.1 in \cite{AliasCN:2021}), we define the smooth function $\varphi\colon [0,+\infty)\rightarrow (0,+\infty)$ by
	$$
	\varphi(t)=\mbox{vol}_M\left(\theta_{t}(B)\right)=\int_{\theta_{t}(B)}dM=\int_{B}\theta^{*}_{t}dM,
	$$
	where $B:=B(p,r)$ is the geodesic ball centered at $p$ with radius $r$.

	Using the compactness of $\overline{B}$ allows us to differentiate under the sign of the integral, we have
	\begin{eqnarray*}\label{eq.ineq_lee}
		\varphi'(t_{0})&=& \dfrac{d}{dt}\bigg|_{t=0}\varphi(t+t_{0})\\
		&=&\dfrac{d}{dt}\bigg|_{t=0}\int_{\theta_{t+t_{0}}(B)}dM\\
		&=&\dfrac{d}{dt}\bigg|_{t=0}\int_{\theta_{t_{0}}(B)}\theta^{*}_{t}dM\nonumber\\
		&=&\int_{\theta_{t_{0}}(B)}\dfrac{d}{dt}\bigg|_{t=0}\theta^{*}_{t}dM \\
		&=&\int_{\theta_{t_{0}}(B)} \mbox{div}(N)dM.
	\end{eqnarray*}
	By Equation \eqref{divN} of the Proposition \ref{propequation}, we have 
	\begin{eqnarray}\label{deriflow}
		\varphi'(t_{0})&=& \int_{\theta_{t_{0}}(B)} (-nH)dM.
	\end{eqnarray}
	
	Now, $c\in [0,+\infty)$ such that $\sup_{p\in M} H(p) = -c \leq 0$. Note that $c=\inf_{x\in M}{|H(x)|}$. Therefore,
	\begin{eqnarray*}\label{eq.desigderive}
		\varphi'(t) &=& \int_{\theta_{t}(B)}\mbox{div}(N)dM\nonumber\\
		& = & \int_{\theta_{t}(B)} (-nH)dM\nonumber\\
		& \geq &(nc)\int_{\theta_{t}(B)}dM\nonumber\\ 
		&=& (nc)\varphi(t),
	\end{eqnarray*}
	for all $t\geq0$. Note that $\varphi'(t)>0,$ for all $t\geq 0$, consequently, $\varphi$ is an increasing function. So, $\varphi(t)\geq \varphi(0)=\mbox{vol}_M(B)>0,$ for all $t\geq0$. Thus, by using the inequality above and integrating the function $\frac{\varphi'(s)}{\varphi(s)}$ over the interval $[0,t]$, we obtain:
	\begin{equation*}\label{Ineq_alias3}
		\int_{0}^{t}\dfrac{\varphi'(s)}{\varphi(s)}ds \geq \int_{0}^{t}(nc)ds.
	\end{equation*}

	Thus, 		
	\begin{equation*}		
		\ln\left(\dfrac{\varphi(t)}{ \mbox{vol}_M(B)}\right) \geq (nc) t.
	\end{equation*}

	Therefore, $\varphi(t)\geq \mbox{vol}_M(B)e^{nc t}$ for all $t\geq 0 $.
	
	Note that $\mbox{vol}_M(B(p,t+r))\geq \mbox{vol}_M(\theta_{t}(B))$, since $\theta_{t}(B)\subset B(p,t+r),$ for all $t\geq0$. Indeed,
	
	\begin{equation}
		\mbox{dist}_M(x,\theta_{t}(x))\leq \int_{0}^{t}\left| \dfrac{d}{ds}\theta_{s}(x) \right|ds=\int_{0}^{t}\| N(\theta_{s}(x))\|ds=t,
	\end{equation}
	where $\mbox{dist}_M(x,\theta_{t}(x))$ the Riemannian distance between $x$ and $\theta_{t}(x)$.
	For all $x\in B$ and by triangular inequality, we have
	$$\mbox{dist}_M(p,\theta_{t}(x))\leq \mbox{dist}_M(x,\theta_{t}(x))+\mbox{dist}_M(p,x)<t+r. $$

	Then 
	$$\mbox{vol}_M(B(p,t+r))\geq \mbox{vol}_M(\theta_{t}(B))=\varphi(t)=\mbox{vol}_M(B)e^{nc t},\,\,\,\forall \,\,t\geq0.$$
Therefore,
$$
\underline{\mu_M}=\liminf\limits_{t\to +\infty}\frac{\ln\mbox{vol}_M\left( B(p,t) \right)}{t}\geq  \liminf\limits_{t\to +\infty}\frac{\ln (\mbox{vol}_M(B)e^{nc t})}{t}=nc.
$$
	
	Finally, if $M$ has zero lower volume entropy, we have that $\underline{\mu_M}=0$.  Thus, $\inf_{x\in M}{|H(x)|}=0$, which finishes the proof of Lemma \ref{lemma:volpol}.
\end{proof}

Therefore, we also have the following equality
$$
nc=\lim\limits_{t\to +\infty}\frac{\ln (\mbox{vol}_M(B)e^{nc t})}{t},
$$
then $nc\leq n\sqrt{K_0}$.

Therefore
$\inf_{x\in M}{|H(x)|}\leq \sqrt{K_0}$.
\end{proof}
In particular, we recover the second main result of Barbosa, Kenmotsu and Oshikiri in \cite{BarbosaKO-1991} (see Theorem 3.8 in \cite{BarbosaKO-1991}).

Note that if $\Sigma^n$ is a Riemannian $n$-manifold that has zero lower volume entropy, then $\Sigma\times \mathbb{R}$ has also zero lower volume entropy. Thus, as a consequence of Lemma \ref{lemma:volpol}, we obtain the following version of \cite[Corollary 1.2]{CoswasckF-2021}.

\begin{corollary}
	Let $\Sigma^n$ be a complete oriented Riemannian $n$-manifold that has zero lower volume entropy. Then, for any graph $\Gamma_f$ over $\Sigma$ that has constant mean curvature $H$, $\Gamma_f$ is minimal and stable hypersurface in $\Sigma\times \mathbb{R}$.
\end{corollary}

We also obtain the following generalisation of Proposition 3.7 in \cite{BarbosaKO-1991}.

\begin{proposition}\label{prop:no_foliation_positive_ricii}
Let $\mathfrak{F}$ be a transversely oriented CMC foliation of an oriented complete Riemannian manifold $M^{n+1}$.   Assume that $M$ is a compact (without boundary) manifold or $\mathfrak{F}$ is an $H$-foliation. Let $N$ be a unit vector field on $M$ orthogonal to $\mathfrak{F}$. Then we have the following:
\begin{enumerate}
    \item [(a)] There is no leaf $L$ of $\mathfrak{F}$ such that $\|\nabla_N N\|\big|_L\in L^1(L)$ and $Ric(N)|_L$ is a non-negative function that is not identically zero. In particular, $\mathfrak{F}$ has no compact leaf $L$ such that $Ric(N)|_L$ is a non-negative function that is not identically zero.
    \item [(b)] If a leaf $L$ of $\mathfrak{F}$ satisfies $\|\nabla_N N\|\big|_L\in L^1(L)$ and $Ric(N)|_L\geq 0$, then $L$ is a totally geodesic leaf and $\|\nabla_N N\|\big|_L\equiv Ric(N)|_L \equiv 0$.
\end{enumerate}
\end{proposition}
\begin{proof}
Let us prove item (a).
Assume by contradiction that there is a leaf $L$ of $\mathfrak{F}$ such that $\|\nabla_N N\|\big|_L\in L^1(L)$ and $Ric(N)|_L$ is a non-negative function that is not identically zero. Then, there is an open subset $U$ of $L$ such that $Ric(N)|_U>0$.

Thus, in points of $L$, we have
$$
\Div_L \nabla_N N=\|\nabla_N N\|^2+\left\|\mathrm{I\!I}\right\|^2+{\rm Ric}(N)\geq 0
$$
and in points of $U$, we have
$$
\Div_L \nabla_N N=\|\nabla_N N\|^2+\left\|\mathrm{I\!I}\right\|^2+{\rm Ric}(N)>0.
$$
By \cite[Proposition 1]{CamargoCS:2010}, $\Div_L \nabla_N N\equiv 0$, which is a contradiction. 

Therefore, there is no such leaf $L$ of $\mathfrak{F}$.

Now, let us prove item (b). Let $L$ be a leaf of $\mathfrak{F}$ that satisfies $\|\nabla_N N\|\big|_L\in L^1(L)$ and $Ric(N)|_L\geq 0$.
Then, in points of $L$, we have
$$
\Div_L \nabla_N N=\|\nabla_N N\|^2+\left\|\mathrm{I\!I}\right\|^2+{\rm Ric}(N)\geq 0.
$$
By \cite[Proposition 1]{CamargoCS:2010}, $\Div_L \nabla_N N\equiv 0$. Then, $\|\nabla_N N\|\equiv \|\mathrm{I\!I}\|\equiv {\rm Ric}(N)\equiv 0$ on $L$.
\end{proof}

\subsection{Proof of Theorem B}

\begin{theorem_B}\label{thm:main_bounded_ricci_compact}
Let $\mathfrak{F}$ be a CMC foliation of a compact (without boundary) Riemannian manifold $M^{n+1}$. Assume that there is $K_0\geq 0$ such that ${\rm Ric}\geq -nK_0$. Suppose that each leaf $L$ of $\mathfrak{F}$ has constant mean curvature $H_L$ such that $|H_L|\geq \sqrt{K_0}$. Then $|H_L|\equiv \sqrt{K_0}$ for any leaf $L$ of $\mathfrak{F}$ and each leaf of $\mathfrak{F}$ is totally umbilical. 
\end{theorem_B}
\begin{proof}
After possibly lifting to the double covering of $M$, we will assume that $M$ is oriented and also that any codimension one CMC foliation of $M$ under consideration is transversely oriented.
Let $N$ be a unit vector field on $M$ orthogonal to $\mathfrak{F}$. 
Let $H\colon M\rightarrow \mathbb{R}$ be the function defined as $H=-\frac{1}{n}\Div (N)$.

Let $L$ be a leaf of $\mathfrak{F}$ such that $|H_L|$ is the maximum of the function $|H|$. 
Changing $H$ by $-H$ and $N$ by $-N$, if necessary, we may assume that $|H_L|$ is the maximum of the function $H$. Thus, $N(H)=0$ on $L$.

Assume by contradiction that $|H_L|> \sqrt{K_0}$. 
Thus, we have two cases to consider:

\noindent Case 1. All the leaves of $\mathfrak{F}$ have the same constant mean curvature $H_L$. In this case, $N(H)\equiv 0$. Then
$$
{\rm div} (\nabla_N N)=\left\|\mathrm{I\!I}\right\|^2+\mbox{Ric}(N).
$$
By the Divergence Theorem, we have
$$
0=\int_M(\left\|\mathrm{I\!I}\right\|^2+\mbox{Ric}(N))\geq \int_M n(H_L^2-K_0)\geq n(H_L^2-K_0){\rm vol} (M)>0,
$$
which is a contradiction.

\noindent Case 2. The leaves of $\mathfrak{F}$ do not have the same constant mean curvature. By Proposition 2.1 in \cite{BarbosaKO-1991}, we may assume that $L$ is a compact leaf. Since $N(H)=0$ on $L$, we have
$$
{\rm div}_L (\nabla_N N)=\left\|\mathrm{I\!I}\right\|^2+\mbox{Ric}(N)+\|\nabla_N N\|^2
$$
on $L$. By the Divergence Theorem, we have
$$
0=\int_L(\left\|\mathrm{I\!I}\right\|^2+\mbox{Ric}(N)+\|\nabla_N N\|^2)\geq \int_L n(H_L^2-K_0)\geq n(H_L^2-K_0){\rm vol} (L)>0,
$$
which is a contradiction.

Thus, in any case, we obtain a contradiction. Therefore, $|H_L|= \sqrt{K_0}$, and thus $|H|\equiv \sqrt{K_0}$. In particular, $N(H)\equiv 0$ on $M$.
Then
$$
{\rm div} (\nabla_N N)=\left\|\mathrm{I\!I}\right\|^2+\mbox{Ric}(N).
$$
By the Divergence Theorem again, we have
$$
0=\int_M(\left\|\mathrm{I\!I}\right\|^2+\mbox{Ric}(N))\geq \int_M n(H_L^2-K_0)\geq n(H_L^2-K_0){\rm vol} (M)=0,
$$
This implies that $\left\|\mathrm{I\!I}\right\|^2\equiv nH_L^2$, $\mbox{Ric}(N)\equiv -nK_0$ and ${\rm div} (\nabla_N N)\equiv 0$ on $M$. By Cauchy's inequality, all the leaves are totally umbilical.
\end{proof}

We obtain the following.
\begin{proposition}
    Let $\Sigma^n$ be a complete oriented Riemannian $n$-manifold with Ricci curvature bounded from below by $-(n-1)K_0$($\leq 0$). Let $f\colon \Sigma\to \mathbb R$ be a smooth function such that its graph $\Gamma_f$ has constant mean curvature $H\geq \sqrt{K_0}$. Assume that the normal vector field $N$ of $\Gamma_f$ satisfies $\|\nabla_N N\|^2\in L^1(\Gamma_f)$. Then $\Gamma_f$ is totally umbilical hypersurface in $\Sigma\times \mathbb{R}$.
\end{proposition}
	
\begin{corollary}
	Let $\Sigma^n$ be a compact (without boundary) oriented Riemannian $n$-manifold with Ricci curvature bounded from below by $-(n-1)K_0$($\leq 0$). Let $f\colon \Sigma\to \mathbb R$ be a smooth function such that its graph has constant mean curvature $H\geq \sqrt{K_0}$. Then $\Gamma_f$ is totally umbilical hypersurface in $\Sigma\times \mathbb{R}$.
\end{corollary}

\subsection{Proof of Theorem C}

In this Subsection, we give a positive answer to Conjecture \ref{conjecture:control_mean_curv_general}, under the additional condition that the ambient space is a compact (without boundary) manifold. 
\begin{theorem_C}\label{thm:main_bounded_sec_curv_compact}
Let $\mathfrak{F}$ be a CMC foliation of a compact (without boundary) Riemannian manifold $M^{n+1}$. Assume that there is $K_0\geq 0$ such that the Ricci curvature of $M$ is bounded from below by $-nK_0$. Then, for each leaf $L$ of $\mathfrak{F}$, the mean curvature $H_L$ of $L$ satisfies $|H_L|\leq \sqrt{K_0}$. Furthermore, if $\mathfrak{F}$  contains a leaf $L$ whose absolute mean curvature is $|H_L|=\sqrt{K_0}$, then either $K_0=0$ and all the leaves of $\mathfrak{F}$ are totally geodesic, or $K_0>0$ and there is a totally umbilical compact leaf in $\mathfrak{F}$.
\end{theorem_C}
\begin{proof}
After possibly lifting to the
two-sheeted covering of $M$, we will assume that $M$ is oriented and also that any codimension one CMC foliation of $M$ under consideration is transversely oriented.

Let $N$ be a unit vector field on $M$ orthogonal to $\mathfrak{F}$. 
Let $H\colon M\rightarrow \mathbb{R}$ be the function defined as $H=-\frac{1}{n}\Div (N)$.

Let $L$ be a leaf of $\mathfrak{F}$ such that $|H_L|$ is the maximum of the function $|H|$. 
Changing $H$ by $-H$ and $N$ by $-N$, if necessary, we may assume that $|H_L|$ is the maximum of the function $H$. Thus, $N(H)=0$ on $L$.

Assume by contradiction that $|H_L|> \sqrt{K_0}$. 

By Theorem B, there is a leaf $\tilde L$ such that $|H_{\tilde L}|<\sqrt{K_0}$. In particular, $H$ is not a constant function. By Proposition 2.1 in \cite{BarbosaKO-1991}, we may assume that $L$ is a compact leaf. Since $N(H)=0$ on $L$, we have
$$
{\rm div}_L (\nabla_N N)=\left\|\mathrm{I\!I}\right\|^2+\mbox{Ric}(N)+\|\nabla_N N\|^2
$$
on $L$. By the Divergence Theorem, we have
$$
0=\int_L(\left\|\mathrm{I\!I}\right\|^2+\mbox{Ric}(N)+\|\nabla_N N\|^2)\geq \int_L n(H_L^2-K_0)\geq n(H_L^2-K_0){\rm vol} (L)>0,
$$
which is a contradiction.

Therefore, $|H_L|\leq \sqrt{K_0}$ for all leaf $L$ of $\mathfrak{F}$.

Now, assume that $\mathfrak{F}$  contains a leaf $L$ whose absolute mean curvature is $|H_L|=\sqrt{K_0}$.

If $K_0=0$, then by Theorem B, all the leaves of $\mathfrak{F}$ are totally geodesic.
If $K_0>0$, we have two cases: (i) $|H|\equiv \sqrt{K_0}$, then by Theorem B again, all the leaves of $\mathfrak{F}$ are totally umbilical; (ii) $H$ is not a constant function, then by Proposition 2.1 in \cite{BarbosaKO-1991}, we may assume that $L$ is a compact leaf, and by proceeding as in the proof of Theorem B, we obtain that $L$ is totally umbilical. 
\end{proof}

\noindent {\bf Acknowledgements.} The authors would like to thank Alexandre Fernandes for his interest in this manuscript, for listening patiently to our first ideas on this research, and for uncountable suggestions.

\end{document}